\documentclass[12pt]{article}
\usepackage[a4paper,margin=1.1in]{geometry}
\usepackage{amssymb}
\usepackage{amsmath}
\usepackage{amsfonts}
\usepackage{amsthm}
\usepackage{color}
\usepackage{float}
\usepackage[all]{xy}
\usepackage{scalerel}
\usepackage{mathabx}
\usepackage{delimset}
\usepackage{setspace}

\newcommand{\qbinom}{\genfrac{[}{]}{0pt}{}}

\newcommand{\minus}{\textup{\texttt{-}}}

\def\P{\mathop{\mbox{\textup{P}}}\nolimits}

\newcommand{\C}{\mathbb{C}}

\newcommand{\s}{\mathrm{S}}
\newcommand{\rr}{\mathrm{R}}

\newcommand{\E}{\mathcal{E}}

\newtheorem{theorem}{Theorem}

\newtheorem{proposition}{Proposition}
\newtheorem{corollary}{Corollary}
\newtheorem{example}{Example}

\newtheorem{remark}{Remark}

\begin{document}

\title{\textbf{Negative-index $u$-deformed homogeneous functions and Ramanujan partial theta functions}}
\author{Ronald Orozco L\'opez}
\newcommand{\Addresses}{{% additional braces for segregating \footnotesize
  \bigskip
  \footnotesize

  %R.~Orozco, \textsc{Departamento de Matematicas, Universidad de los Andes, 
  % Carrera 1 N. 18A-12 Bogot\'a, Colombia}\par\nopagebreak
  \textit{E-mail address}, R.~Orozco: \texttt{rj.orozco@uniandes.edu.co}
  
}}

\maketitle
%\tableofcontents

\begin{abstract}
We introduce and study a class of \(u\)-deformed homogeneous functions of complex order \(\rr_{\alpha}(x,y;u\,|\,q)\), obtained by extending \(u\)-deformed homogeneous polynomials beyond the nonnegative integer regime. We establish their basic analytic and algebraic properties, including convergence criteria, recurrence relations, \(q\)-derivative formulas, and an operator-theoretic realization through a deformed \(q\)-exponential operator. A central feature of the theory appears in the negative-index sector: for every nonnegative integer \(n\), the functions \(\rr_{\minus n\minus1}(x,y;u\,|\,q)\) admit finite triangular decompositions in terms of shifted Ramanujan partial theta functions. We also prove the inverse triangular transformation, showing that the corresponding partial theta functions can be recovered from the negative-index \(u\)-deformed homogeneous functions. We also derive pantograph-type \(q\)-difference equations for these functions. Finally, we derive specialized forms of the negative-index theory in Cauchy and Stieltjes--Wigert-type regimes, including a boundary geometric case and genuine partial-theta regimes.
\end{abstract}
\noindent 2020 {\it Mathematics Subject Classification}:
Primary 33D15. Secondary 05A30; 33D45; 39A13.

\noindent \emph{Keywords: } \(u\)-deformed homogeneous functions; complex-order \(q\)-series; basic hypergeometric series; generalized \(q\)-binomial coefficients; Ramanujan partial theta functions; \(q\)-difference equations; pantograph-type equations; Stieltjes--Wigert-type functions.

\section{Introduction}

The theory of \(q\)-series and basic hypergeometric functions provides a central framework for the study of special functions, \(q\)-difference equations, and combinatorial generating functions. Among its fundamental objects are \(q\)-shifted factorials, generalized \(q\)-binomial coefficients, Jackson \(q\)-derivatives, and \(q\)-exponential operators. These tools make it possible to construct and relate families of \(q\)-polynomials and analytic \(q\)-series through operator-theoretic and algebraic methods. A recurring question in this context is whether polynomial families indexed by nonnegative integers admit meaningful extensions to complex order. In the classical binomial setting, the identity 
\[ 
(x+y)^n = \sum_{k=0}^{n} \binom{n}{k}x^{n-k}y^k 
\] 
extends naturally to complex order through a convergent binomial series inside a nonzero disk. The corresponding \(q\)-binomial situation is more delicate. Replacing the integer degree by a complex parameter in the natural \(q\)-binomial expansion may destroy convergence: the resulting non-polynomial \(q\)-series can have radius of convergence zero. This obstruction shows that a direct \(q\)-binomial continuation does not always produce an analytic object. The purpose of this paper is to introduce a deformation mechanism that controls this obstruction. We study the \(u\)-deformed homogeneous functions 
\[ 
\rr_{\alpha}(x,y;u\,|\,q) = \sum_{k=0}^{\infty} \begin{bmatrix} \alpha\\ k \end{bmatrix}_{q} u^{\binom{k}{2}}x^{\alpha-k}y^k, \qquad 0<q<1. 
\] 
For \(\alpha=n\in\mathbb N_0\), this series terminates and recovers the \(u\)-deformed homogeneous polynomials \cite{orozco}. For general \(\alpha\in\mathbb C\), however, the same expression gives a complex-order continuation whose analytic behavior depends essentially on the deformation parameter \(u\). A first analytic feature of the theory is the presence of a sharp threshold at \(|u|=q\). Below this threshold, the complex-order extensions are entire in the second variable; at the threshold, they converge in a finite disk; above it, the non-polynomial extensions have radius of convergence zero. Thus the parameter \(u\) acts as a regularizing deformation: it determines whether the formal complex-order \(q\)-binomial expansion becomes a genuine analytic family or remains only a divergent formal object. The main structural phenomenon of the paper appears in the negative-index sector. While nonnegative integer orders give finite polynomials, negative integer orders produce infinite \(q\)-series. We show that these infinite series are not arbitrary: for every \(n\in\mathbb N_0\), the function 
\[ 
\rr_{\minus n\minus1}(x,y;u\,|\,q) 
\] 
can be expressed as a finite triangular combination of shifted Ramanujan partial theta functions. Moreover, this triangular relation is invertible: the corresponding shifted partial theta functions can be recovered from the negative-index \(u\)-deformed homogeneous functions. Thus the negative-index sector is triangularly equivalent to a finite system of shifted partial theta functions. 

The paper is organized as follows. In Section~2 we collect the preliminaries on \(q\)-shifted factorials, generalized \(q\)-binomial coefficients, \(u\)-deformed basic hypergeometric series, and Jackson \(q\)-derivatives. Section~3 introduces the functions \(\rr_{\alpha}(x,y;u\,|\,q)\), establishes their first representations, discusses their structural interpretation, proves convergence properties, and derives recurrence and \(q\)-differentiation formulas. In particular, the convergence theorem identifies the threshold \(|u|=q\) and explains the obstruction to the undeformed complex-order continuation. Section~4 provides an operator-theoretic realization of \(\rr_{\alpha}(x,y;u\,|\,q)\) through the action of a \(u\)-deformed \(q\)-exponential operator on the complex power \(x^\alpha\). This shows that the functions are not merely formal series, but arise as deformed \(q\)-translation orbits of complex powers. Section~5 uses the recurrence and \(q\)-differentiation formulas from Section~3 to derive pantograph-type \(q\)-difference equations for suitable specializations of \(\rr_{\alpha}(x,y;u\,|\,q)\). These equations couple proportional scales such as \(y\), \(qy\), \(uy\), and \(uqy\). Section~6 contains the central results of the paper. We introduce Ramanujan's partial theta function and prove a finite theta decomposition for the negative-index functions \(\rr_{\minus n\minus1}(x,y;u\,|\,q)\). We then prove the inverse triangular transformation, showing that the shifted partial theta functions and the negative-index \(u\)-deformed homogeneous functions determine each other. The section also includes explicit low-order examples and an asymptotic limit obtained when the negative index tends to infinity under a simultaneous \(q\)-scaling of the second variable. Finally, Section~7 develops several specializations. The Cauchy specialization corresponds to the boundary case \(u=q\), where the partial theta parameter becomes \(1\) and the partial theta functions degenerate into geometric series. This produces finite rational decompositions in the negative-index sector. The Stieltjes--Wigert-type specializations correspond to \(u=q^b\), \(b\geq2\), and lead to genuine partial theta functions with parameter \(q^{b-1}\). We also record the half-integer Exton-type regime \(u=q^{b+\frac12}\), whose theta parameter is \(q^{b-\frac12}\). These specializations show that the triangular correspondence obtained in Section~6 is stable under classical \(q\)-special regimes and is not tied to a single value of the deformation parameter.

\section{Preliminaries}

In this section, we collect the notation and basic identities used throughout the paper. We recall \(q\)-shifted factorials, generalized \(q\)-binomial coefficients, the \(u\)-deformed basic hypergeometric series, and the Jackson \(q\)-derivative. These preliminaries provide the algebraic and analytic framework for the construction of the complex-order \(u\)-deformed homogeneous functions studied below.
\subsection{$q$-shifted factorials and $q$-binomial coefficients}
Some notation and terminology for basic hypergeometric series are taken from \cite{gasper}. Let $0<q<1$ and the $q$-shifted factorial be defined by
\begin{align*}
    (a;q)_{n}&=\prod_{k=0}^{n\minus1}(1-q^{k}a),\\
    (a;q)_{\infty}&=\lim_{n\rightarrow\infty}(a;q)_{n}=\prod_{k=0}^{\infty}(1-aq^{k}).
\end{align*}
The multiple $q$-shifted factorials are defined by
\begin{align*}
    (a_{1},a_{2},\ldots,a_{m};q)_{n}&=(a_{1};q)_{n}(a_{2};q)_{n}\cdots(a_{m};q)_{n},\\
    (a_{1},a_{2},\ldots,a_{m};q)_{\infty}&=(a_{1};q)_{\infty}(a_{2};q)_{\infty}\cdots(a_{m};q)_{\infty}.
\end{align*}
Some useful identities for $q$-shifted factorial:
\begin{align}
    (a;q)_{n}&=\sum_{k=0}^{n}\qbinom{n}{k}_{q}(\minus a)^{k}q^{\binom{k}{2}},\label{iden3}\\
    (a;q)_{n}&=\frac{(a;q)_{\infty}}{(aq^n;q)_{\infty}},\label{iden4}\\
    (a;q)_{n+k}&=(a;q)_{n}(aq^{n};q)_{k},\label{iden5}\\
    (aq^{\minus n};q)_{n}&=(\minus a)^nq^{\minus\binom{n+1}{2}}(a^{\minus1}q;q)_{n},\label{iden6}
\end{align}
where the $q$-binomial coefficient is defined by
\begin{equation*}
\qbinom{n}{k}_{q}=\frac{(q;q)_{n}}{(q;q)_{k}(q;q)_{n\minus k}}.
\end{equation*}
The generalized $q$-binomial coefficient verifies that:
\begin{align}
    \qbinom{\alpha+1}{k}_{q}&=\qbinom{\alpha}{k}_{q}+q^{\alpha+1\minus k}\qbinom{\alpha}{k\minus1}_{q}=
    q^{k}\qbinom{\alpha}{k}_{q}+\qbinom{\alpha}{k\minus1}_{q},\label{eqn_pascal}\\
    \qbinom{\alpha}{k}_{q}&=\frac{(q^{\minus\alpha};q)_{k}}{(q;q)_{k}}(\minus q^\alpha)^{k}q^{\minus\binom{k}{2}},\label{iden1}\\
    \qbinom{\minus\alpha}{k}_{q}&=\qbinom{\alpha+k\minus1}{k}_{q}(\minus q^{\minus\alpha})^kq^{\minus\binom{k}{2}}.\label{iden2}
\end{align}
In our work, we will use the identities for binomial coefficients:
\begin{align}
    \binom{n+k}{2}&=\binom{n}{2}+\binom{k}{2}+nk,\label{iden7}\\
    \binom{n-k}{2}&=\binom{n}{2}+\binom{k}{2}+k(1-n).\label{iden8}
\end{align}

\subsection{Deformed basic hypergeometric series}

In \cite{orozco}, a class of \(u\)-deformed basic hypergeometric series depending on an additional deformation parameter was introduced. Their convergence properties and their relation with classical basic hypergeometric series under suitable specializations of \(u\) were also discussed. This ${}_{r}\Phi_{s}$-series is
   \begin{align}
        &{}_{r}\Phi_{s}\left(
    \begin{array}{c}
         a_{1},a_{2},\ldots,a_{r} \\
         b_{1},\ldots,b_{s}
    \end{array}
    ;q,u,z
    \right)\nonumber\\
    &\hspace{3cm}=\sum_{n=0}^{\infty}u^{\binom{n}{2}}\frac{(a_{1},a_{2},\ldots,a_{r};q)_{n}}{(q,b_{1},b_{2},\ldots,b_{s};q)_{n}}\bigg[(\minus1)^nq^{\binom{n}{2}}\bigg]^{1+s-r}z^n,\label{def_DBHS}
    \end{align}
where $0<q<1$ and $u\in\C$. If we set \(u=1\) in Eq.~(\ref{def_DBHS}), we recover the classical basic hypergeometric series
\begin{equation*}
    {}_r\phi_{s}\left(
    \begin{array}{c}
         a_{1},a_{2},\ldots,a_{r} \\
         b_{1},\ldots,b_{s}
    \end{array}
    ;q,z
    \right)=\sum_{n=0}^{\infty}\frac{(a_{1},a_{2},\ldots,a_{r};q)_{n}}{(q;q)_{n}(b_{1},b_{2},\ldots,b_{s};q)_{n}}\Big[(\minus1)^{n}q^{\binom{n}{2}}\Big]^{1+s-r}z^n.
\end{equation*}
In this paper, we will frequently use the $q$-binomial theorem:
\begin{equation}
    {}_1\phi_{0}\left(\begin{array}{c}
         a\\
         - 
    \end{array};q,z\right)=\frac{(az;q)_{\infty}}{(z;q)_{\infty}}=\sum_{n=0}^{\infty}\frac{(a;q)_{n}}{(q;q)_{n}}z^{n}.
\end{equation}

Some convergence conditions for $_{r}\Phi_{s}$-series are
\begin{itemize}
    \item $1+s-r\neq0$. If $0<\vert uq^{1+s-r}\vert<1$, then ${}_{r}\Phi_{s}$ is an entire function. If $\vert uq^{1+s-r}\vert=1$, then ${}_{r}\Phi_{s}$ converges for $\vert z\vert<1$. If $\vert uq^{1+s-r}\vert>1$, then ${}_{r}\Phi_{s}$ is divergent.
    \item $1+s-r=0$. If $0<\vert u\vert<1$, then ${}_{r}\Phi_{s}$ is an entire function. If $\vert u\vert=1$, then ${}_{r}\Phi_{s}$ converges for $\vert z\vert<1$. If $\vert u\vert>1$, then ${}_{r}\Phi_{s}$ is divergent.
\end{itemize}
For the special choices \(u=q^b\) and \(u=q^{b+\frac12}\), one obtains the following reductions to classical basic hypergeometric series:
\begin{itemize}
    \item If $u=q^b$, $b\geq0$, 
\begin{equation}
    {}_{1+r}\Phi_{r}\left(
    \begin{array}{c}
         a_{1},a_{2},\ldots,a_{r+1}\\
         b_{1},\ldots,b_{r}
    \end{array}
    ;q,q^b,z
    \right)={}_{1+r}\phi_{r+b}\left(
    \begin{array}{c}
         a_{1},a_{2},\ldots,a_{r+1} \\
         b_{1},b_{2},\ldots,b_{r},\mathbf{0}_{b}
    \end{array}
    ;q,(\minus1)^bz
    \right)
\end{equation}
for all $z\in\C$.
\item If $u=q^{b+1/2}$, $b\geq0$
\begin{align}
    &{}_{r+1}\Phi_{r}\left(
    \begin{array}{c}
         a_{1},\ldots,a_{r+1} \\
         b_{1},\ldots,b_{r}
    \end{array}
    ;q,q^{b+1/2},z
    \right)\nonumber\\
    &\hspace{1cm}={}_{2r+2}\phi_{2(r+b)+2}\left(
    \begin{array}{ccc}
         \sqrt{a_{1}},-\sqrt{a_{1}}&,\ldots,&\sqrt{a_{r+1}},-\sqrt{a_{r+1}}\hspace{0.6cm}\\
         \sqrt{b_{1}},-\sqrt{b_{1}}&,\ldots,&\sqrt{b_{r}},-\sqrt{b_{r}},-\sqrt{q},\mathbf{0}_{2b+1}
    \end{array}
    ;\sqrt{q},\minus z
    \right)
\end{align}
for all $z\in\C$. 
\end{itemize}

\subsection{The Jackson $q$-derivative}

The Jackson \(q\)-derivative is defined by
\begin{equation*}
    D_{q}f(x)=\frac{f(x)-f(qx)}{x}
\end{equation*}
and the Leibniz rule for $D_{q}$
\begin{equation}\label{eqn_leibniz}
    D_{q}^{n}\{f(x)g(x)\}=\sum_{k=0}^{n}q^{k(k\minus n)}\qbinom{n}{k}_{q}D_{q}^{k}\{f(x)\}D_{q}^{n\minus k}\{g(q^{k}x)\}.
\end{equation}
From Eq.(\ref{iden2}), for $k\geq1$,
\begin{equation}
    D_{q}^nx^{\minus k}=(q;q)_{n}\qbinom{-k}{n}_{q}x^{-n-k}=(\minus1)^{n}q^{\minus kn\minus\binom{n}{2}}(q;q)_{n}\qbinom{n+k-1}{n}_{q}x^{\minus k\minus n}.\label{iden9}
\end{equation}
Mansour at. el. \cite{mansour} showed that the \(k\)-th \(q\)-derivative of a function \(f\) can be expressed in terms of its values at the points \(xq^j\), \(j=0,1,\ldots,k\), as follows:
\begin{equation}\label{mansour}
    D_{q}^{k}f(x)=(-1)^kx^{-k}q^{-\binom{k}{2}}\sum_{i=0}^{k}(-1)^i\qbinom{k}{i}_{q}q^{\binom{i}{2}}f(xq^{k-i}).
\end{equation}

\section{Deformed homogeneous functions $\rr_{\alpha}(x,y;u\,|\,q)$}

This section is devoted to the definition and study of the $u$‑deformed homogeneous functions $\rr_{\alpha}(x,y;u\,|\,q)$. We present their main analytic and algebraic properties, including convergence results, recurrence relations, and $q$‑derivatives. Particular attention is given to their behavior for negative integer indices and their structural connections with other classes of $q$‑special functions.

\subsection{Definition and first representations}

Throughout the paper, we assume \(0<q<1\). For \(\alpha\in\mathbb{C}\),
\(u\in\mathbb{C}\), and \(x\neq 0\), we define the \(u\)-deformed
homogeneous function of order \(\alpha\) by
\begin{equation}\label{def_function}
\rr_{\alpha}(x,y;u\,|\,q)
=
\sum_{k=0}^{\infty}
\begin{bmatrix}
\alpha\\ k
\end{bmatrix}_{q}
u^{\binom{k}{2}}
x^{\alpha-k}y^k .
\end{equation}

When \(\alpha=n\in\mathbb{N}_0\), the above series terminates and one
recovers the \(u\)-deformed homogeneous polynomials
\[
\rr_n(x,y;u\,|\,q)
=
\sum_{k=0}^{n}
\begin{bmatrix}
n\\ k
\end{bmatrix}_{q}
u^{\binom{k}{2}}
x^{n-k}y^k.
\]
Thus \(\rr_{\alpha}(x,y;u\,|\,q)\) provides a continuation of the polynomial
family from nonnegative integer order to arbitrary complex order.
Using Eq.(\ref{iden1}) we obtain
\begin{equation}\label{qseries_rr}
\rr_{\alpha}(x,y;u\,|\,q)
=
x^\alpha
\sum_{k=0}^{\infty}\left(\frac{u}{q}\right)^{\binom{k}{2}}
\frac{(q^{-\alpha};q)_k}{(q;q)_k}
\left(-q^\alpha\frac{y}{x}\right)^k .
\end{equation}
In particular, for \(\alpha\neq 0\), the negative-order functions admit
the expansion
\begin{equation}\label{rr_GF_neg}
\rr_{-\alpha}(x,y;u\,|\,q)
=
\frac{1}{x^\alpha}
\sum_{k=0}^{\infty}
\begin{bmatrix}
\alpha+k-1\\ k
\end{bmatrix}_{q}
\left(\frac{u}{q}\right)^{\binom{k}{2}}
\left(-\frac{y}{q^\alpha x}\right)^k .
\end{equation}
This form shows that negative orders are governed by the generating
functions of generalized \(q\)-binomial coefficients. This negative-order expansion will be the starting point for the finite theta decompositions obtained in Section 6.

\begin{proposition}[Deformed hypergeometric representation]
Let \(0<q<1\), \(u\in\mathbb C\), \(\alpha\in\mathbb C\), and \(x\neq 0\).
Then
\[
\rr_{\alpha}(x,y;u\,|\,q)
=
x^\alpha\,
{}_2\Phi_0
\left(
\begin{matrix}
q^{-\alpha},0\\
-
\end{matrix}
;q,u,q^\alpha\frac{y}{x}
\right).
\]
\end{proposition}
\begin{proof}
By the definition of the deformed basic hypergeometric series,
\[
{}_2\Phi_0
\left(
\begin{matrix}
q^{-\alpha},0\\
-
\end{matrix}
;q,u,q^\alpha\frac{y}{x}
\right)
=
\sum_{n=0}^{\infty}
u^{\binom{n}{2}}
\frac{(q^{-\alpha};q)_n(0;q)_n}{(q;q)_n}
\left[(-1)^nq^{\binom{n}{2}}\right]^{-1}
\left(q^\alpha\frac{y}{x}\right)^n .
\]
Multiplying by \(x^\alpha\), we get
\[
x^\alpha
{}_2\Phi_0
\left(
\begin{matrix}
q^{-\alpha},0\\
-
\end{matrix}
;q,u,q^\alpha\frac{y}{x}
\right)
=
x^\alpha
\sum_{n=0}^{\infty}\left(\frac{u}{q}\right)^{\binom{n}{2}}
\frac{(q^{-\alpha};q)_n}{(q;q)_n}
\left(-q^\alpha\frac{y}{x}\right)^n,
\]
which from Eq.(\ref{qseries_rr}) is precisely \(\rr_{\alpha}(x,y;u\,|\,q)\).
\end{proof}

\subsection{Structural interpretation}

The family \(\rr_{\alpha}(x,y;u\,|\,q)\) should be regarded not merely as a deformation of homogeneous polynomials, but as a two-parameter mechanism by which binomial-type algebraic structures are continued beyond the polynomial regime. In the classical and \(q\)-binomial settings, homogeneity is governed by linear shifts of the index and by first-order \(q\)-difference symmetries. The additional quadratic weight \(u^{\binom{k}{2}}\) introduces a second layer of deformation, encoding a nonlinear dependence on the summation index and thereby altering the analytic nature of the resulting expansion.

From this point of view, the passage from nonnegative integer order to
arbitrary complex order is not a purely formal extension, but a structural
transition. For \(\alpha\in\mathbb{N}_0\), the functions
\(\rr_{\alpha}(x,y;u\,|\,q)\) are finite algebraic objects. For nonintegral and,
in particular, negative integral orders, they become genuinely analytic
objects governed by infinite \(q\)-series. As will be shown in Section 6,
the negative-index sector is controlled by Ramanujan partial theta
functions through a finite triangular correspondence. Thus the same
deformed homogeneous calculus interpolates between finite \(q\)-binomial
geometry and theta-type analytic phenomena.

\subsection{Convergence}

We next record the basic convergence properties of
\(\rr_{\alpha}(x,y;u\,|\,q)\). These properties already show that the
deformation parameter \(u\) plays a decisive analytic role.

\begin{theorem}[Convergence]\label{conver}
Let \(0<q<1\), \(x\neq 0\), \(u\in\mathbb C\), and
\(\alpha\in\mathbb C\). We use the convention \(u^0=1\), also when
\(u=0\).

\begin{enumerate}
\item If \(\alpha\in\mathbb N_0\), then
\(\rr_{\alpha}(x,y;u\,|\,q)\) is a polynomial in \(y\). In particular,
it is entire in \(y\) for every \(u\in\mathbb C\).

\item If \(\alpha\notin\mathbb N_0\), then:
\begin{enumerate}
\item If \(|u|<q\), then \(\rr_{\alpha}(x,y;u\,|\,q)\) is entire in \(y\).
\item If \(|u|=q\), then \(\rr_{\alpha}(x,y;u\,|\,q)\) converges for
\[
\left|\frac{y}{x}\right|<|q^{-\alpha}|.
\]
\item If \(|u|>q\), then the radius of convergence in \(y\) is zero.
\end{enumerate}
\end{enumerate}
\end{theorem}
\begin{proof}
If \(\alpha\in\mathbb N_0\), then
\(\begin{bmatrix}\alpha\\ k\end{bmatrix}_q=0\) for \(k>\alpha\), so the
series terminates. If \(u=0\), then \(u^{\binom{k}{2}}=0\) for \(k\geq 2\), while \(u^0=1\) by convention. Hence the defining series reduces to its first two terms and is entire in \(y\). Assume now that \(u\neq 0\) and \(\alpha\notin\mathbb N_0\). Using the representation
\[
\rr_{\alpha}(x,y;u\,|\,q)
=
x^\alpha
\sum_{k=0}^{\infty}
\frac{(q^{-\alpha};q)_k}{(q;q)_k}
\left(\frac{u}{q}\right)^{\binom{k}{2}}
\left(q^\alpha\frac{y}{x}\right)^k,
\]
the ratio of two consecutive terms has modulus asymptotic to
\[
\left|q^\alpha\frac{y}{x}\right|
\left|\frac{u}{q}\right|^k ,
\qquad k\to\infty.
\]
Hence the limit is \(0\) when \(|u|<q\), equals
\(\left|q^\alpha y/x\right|\) when \(|u|=q\), and is infinite when
\(|u|>q\), unless \(y=0\). The result follows from the ratio test.
\end{proof}

\begin{remark}[Obstruction to undeformed complex-order continuations]
Let
\[
\mathrm{H}_{\alpha}(x,y\,|\,q)
=
\sum_{k=0}^{\infty}
\begin{bmatrix}
\alpha\\ k
\end{bmatrix}_q
x^{\alpha-k}y^k
\]
be the formal \(q\)-binomial continuation of the two-variable
Rogers--Szegő polynomials. This corresponds to the specialization
\(u=1\) of \(\rr_{\alpha}(x,y;u\,|\,q)\). Since \(0<q<1\), we have
\(|u|=1>q\). Hence, by Theorem~\ref{conver}, if
\(\alpha\notin\mathbb N_0\), the series \(H_{\alpha}(x,y\,|\,q)\)
has radius of convergence zero in the variable \(y\).

This behavior contrasts with the classical binomial expansion
\[
(x+y)^\alpha
=
x^\alpha
\sum_{k=0}^{\infty}
\binom{\alpha}{k}
\left(\frac{y}{x}\right)^k,
\qquad
\left|\frac{y}{x}\right|<1,
\]
which provides a nontrivial analytic continuation of the polynomial
identity \((x+y)^n\) to complex order. Thus, while the classical
binomial coefficients allow a convergent complex-order continuation
inside a nonzero disk, the natural \(q\)-binomial continuation fails to
do so in the undeformed case.

More generally, the same analytic obstruction occurs for every
deformation parameter satisfying \(|u|>q\). In this regime, the
quadratic factor \(u^{\binom{k}{2}}\) does not compensate the growth
coming from the generalized \(q\)-binomial coefficients, and the
resulting complex-order series has radius of convergence zero in \(y\). Thus the parameter \(u\) is not merely auxiliary; it determines whether the complex-order continuation is an analytic object or only a divergent formal expansion.

\end{remark}

\begin{corollary}
Let \(m\in\mathbb N\). Then \(\rr_{-m}(x,y;u\,|\,q)\) is entire in \(y\)
whenever \(|u|<q\). If \(|u|=q\), then it converges for
\[
\left|\frac{y}{x}\right|<q^m.
\]
If \(|u|>q\), then its radius of convergence in \(y\) is zero.
\end{corollary}

\subsection{Recurrence relations and \(q\)-derivatives}

The functions \(\rr_{\alpha}(x,y;u\,|\,q)\) retain several of the structural
features of homogeneous polynomials. In particular, they satisfy
binomial-type recurrence relations and simple \(q\)-differentiation
formulas. In what follows, \(D_q\) acts with respect to the variable \(x\).

\begin{proposition}
For every \(\alpha\in\mathbb C\), the following identities hold as formal power series, and analytically wherever the corresponding series
converge.
\[
\rr_{\alpha+1}(x,y;u\,|\,q)
=
x\rr_{\alpha}(x,qy;u\,|\,q)
+
y\rr_{\alpha}(x,uy;u\,|\,q),
\]
and
\[
\rr_{\alpha+1}(x,y;u\,|\,q)
=
x\rr_{\alpha}(x,y;u\,|\,q)
+
y\rr_{\alpha}(qx,uy;u\,|\,q).
\]
Moreover,
\[
D_q\{\rr_{\alpha}(x,a;u\,|\,q)\}
=
(1-q^\alpha)\rr_{\alpha-1}(x,a;u\,|\,q),
\]
and
\[
D_q\{\rr_{\alpha}(a,x;u\,|\,q)\}
=
(1-q^\alpha)\rr_{\alpha-1}(a,ux;u\,|\,q).
\]
For \(k\geq 1\),
\begin{equation}\label{derkx_rr}
D_q^k\{\rr_{\alpha}(x,a;u\,|\,q)\}
=
(q;q)_k
\begin{bmatrix}
\alpha\\ k
\end{bmatrix}_q
\rr_{\alpha-k}(x,a;u\,|\,q),
\end{equation}
and
\[
D_q^k\{\rr_{\alpha}(a,x;u\,|\,q)\}
=
u^{\binom{k}{2}}
(q;q)_k
\begin{bmatrix}
\alpha\\ k
\end{bmatrix}_q
\rr_{\alpha-k}(a,u^kx;u\,|\,q).
\]
\end{proposition}
\begin{proof}
The proof is the same algebraic argument as in the polynomial case,
with the ordinary \(q\)-binomial coefficients replaced by the generalized
coefficients \(\begin{bmatrix}\alpha\\ k\end{bmatrix}_q\). The two recurrence
relations follow from the two \(q\)-Pascal identities Eq.(\ref{eqn_pascal}). Indeed, substituting the defining series and reindexing the second sums gives the stated identities.

For the \(q\)-derivative formulas, we apply
\[
D_q x^\beta=(1-q^\beta)x^{\beta-1}
\]
term by term to the defining series. This is justified in the domain of
convergence described in Theorem~\ref{conver}; equivalently,
the identities may first be read as identities of formal power series.
The elementary relation
\[
(1-q^{\alpha-k})
\begin{bmatrix}
\alpha\\ k
\end{bmatrix}_q
=
(1-q^\alpha)
\begin{bmatrix}
\alpha-1\\ k
\end{bmatrix}_q
\]
then gives
\[
D_q\{\rr_\alpha(x,a;u\,|\,q)\}
=
(1-q^\alpha)\rr_{\alpha-1}(x,a;u\,|\,q).
\]
The formula for \(D_q\{\rr_{\alpha}(a,x;u\,|\,q)\}\) follows similarly:
after reindexing, the factor \(u^j\) in the \(j\)-th term is absorbed
by replacing \(x\) with \(ux\) in \(\rr_{\alpha-1}(a,x;u\,|\,q)\).
Iterating the first-order identities yields the stated formulas for
\(D_q^k\); in the second case the successive shifts produce the factor
\(u^{\binom{k}{2}}\).
\end{proof}

These identities show that the family \(\rr_{\alpha}(x,y;u\,|\,q)\) behaves
as a deformed homogeneous system under the Jackson \(q\)-derivative. In
the next section we explain this phenomenon from an operator-theoretic
point of view.

\section{Operator-theoretic realization}

We now show that the functions introduced in the previous section arise
naturally from the action of a deformed \(q\)-exponential operator on
complex powers. This operator-theoretic point of view is important
because it shows that the functions \(\rr_{\alpha}(x,y;u\,|\,q)\) are not defined merely by a formal series, but are generated by a deformation of the classical \(q\)-translation mechanism.

Let \(D_{q,x}\) denote the Jackson \(q\)-derivative with respect to \(x\). Following \cite{orozco}, we define the \(u\)-deformed \(q\)-exponential operator by
\[
\E(yD_{q,x}\,|\,u)
=
\sum_{n=0}^{\infty}
u^{\binom{n}{2}}
\frac{(yD_{q,x})^n}{(q;q)_n}.
\]
When \(u=0\), we use the convention \(u^0=1\), so that the above series
reduces to
\[
\E(yD_{q,x}\,|\,0)
=
1+\frac{yD_{q,x}}{1-q}.
\]

Thus \(\rr_{\alpha}(x,y;u\,|\,q)\) may be viewed as the deformed \(q\)-translation orbit of the complex power \(x^\alpha\).

\begin{proposition}[Operator realization]\label{prop_translation}
Let \(0<q<1\), \(u\in\mathbb C\), \(\alpha\in\mathbb C\), and \(x\neq0\).
Then
\[
\E(yD_{q,x}\,|\,u)\{x^\alpha\}
=
\rr_{\alpha}(x,y;u\,|\,q),
\]
as a formal power series in \(y\), and analytically wherever the series
on the right-hand side converges.
\end{proposition}
\begin{proof}
Using
\[
D_{q,x}^{\,n}x^\alpha
=
(q;q)_n
\begin{bmatrix}
\alpha\\ n
\end{bmatrix}_q
x^{\alpha-n},
\]
we obtain
\[
\mathcal E(yD_{q,x}\,|\,u)\{x^\alpha\}
=
\sum_{n=0}^{\infty}
u^{\binom{n}{2}}
\frac{y^n}{(q;q)_n}
D_{q,x}^{\,n}x^\alpha.
\]
Therefore
\[
\mathcal E(yD_{q,x}\,|\,u)\{x^\alpha\}
=
\sum_{n=0}^{\infty}
\begin{bmatrix}
\alpha\\ n
\end{bmatrix}_q
u^{\binom{n}{2}}
x^{\alpha-n}y^n
=
\rr_{\alpha}(x,y;u\,|\,q).
\]
The identity is formal in \(y\), and it is analytic in the convergence
region described in Theorem~1.
\end{proof}

\begin{remark}
For \(\alpha\in\mathbb N_0\), the above identity reduces to the operator representation of the \(u\)-deformed homogeneous polynomials. For general complex \(\alpha\), the same operator generates the complex-order continuation. Thus the passage from polynomials to non-polynomial \(q\)-series is realized by applying the same deformed \(q\)-exponential operator to the family of complex powers \(x^\alpha\). This provides an operator-theoretic interpretation of the functions \(\rr_\alpha(x,y;u\,|\,q)\), complementing the recurrence and \(q\)-differentiation properties established in Proposition~2.
\end{remark}

\section{Pantograph-type $q$-difference equations}

We now use the recurrence and \(q\)-differentiation formulas obtained in Proposition~2 to derive pantograph-type \(q\)-difference equations for suitable specializations of \(\rr_\alpha(x,y;u\,|\,q)\).

\begin{theorem}
Let \(u\in\mathbb C^\ast\) and \(\alpha\in\mathbb C\). For \(x\neq0\), define
\[
F_\alpha(x,y)=\rr_{\alpha}(x,y;u\,|\,q).
\]
Then \(F_\alpha\) satisfies the functional equations of pantograph type
\begin{align}
x(D_{q,x}F_\alpha)(x,qy)
+
y(D_{q,x}F_\alpha)(x,uy)
&=
(1-q^\alpha)F_\alpha(x,y),\label{eqn_gpfde2}
\end{align}
and
\begin{equation}\label{eqn_gpfde}
x(D_{q,y}F_\alpha)(x,u^{-1}y)
+
q^{\alpha-1}y(D_{q,y}F_\alpha)(x,q^{-1}y)
=
(1-q^\alpha)F_\alpha(x,y)    
\end{equation}
with initial value
\[
F_\alpha(x,0)=x^\alpha.
\]
The identities hold as formal identities in the expansion variable and analytically wherever the corresponding series converge.
\end{theorem}
\begin{proof}
The first identity follows directly from the first recurrence relation in Proposition~2 together with the formula for \(D_{q,x}\). We prove the second identity, since it gives the proportional equation used below. By Proposition~2,
\[
D_{q,y}F_\alpha(x,y)
=
(1-q^\alpha)\rr_{\alpha-1}(x,uy;u\,|\,q).
\]
Hence
\[
(D_{q,y}F_\alpha)(x,u^{-1}y)
=
(1-q^\alpha)\rr_{\alpha-1}(x,y;u\,|\,q),
\]
and
\[
(D_{q,y}F_\alpha)(x,q^{-1}y)
=
(1-q^\alpha)\rr_{\alpha-1}(x,uq^{-1}y;u\,|\,q).
\]
Therefore the left-hand side is
\[
(1-q^\alpha)
\left[
x\rr_{\alpha-1}(x,y;u\,|\,q)
+
q^{\alpha-1}y\rr_{\alpha-1}(x,uq^{-1}y;u\,|\,q)
\right].
\]
Expanding the two terms gives
\[
x\rr_{\alpha-1}(x,y;u\,|\,q)
=
\sum_{k=0}^{\infty}\qbinom{\alpha-1}{k}_q
u^{\binom{k}{2}}x^{\alpha-x}y^k
\]
and
\[
q^{\alpha-1}y\rr_{\alpha-1}(x,uq^{-1}y;u\,|\,q)
=
\sum_{k=1}^{\infty}
q^{\alpha-k}\qbinom{\alpha-1}{k-1}_q
u^{\binom{k}{2}}x^{\alpha-k}y^k.
\]
Thus the expression in brackets is
\[
x^\alpha+
\sum_{k=1}^{\infty}
\left(\qbinom{\alpha-1}{k}_q+
q^{\alpha-k}\qbinom{\alpha-1}{k-1}_q
\right)
u^{\binom{k}{2}}x^{\alpha-k}y^k.
\]
By the \(q\)-Pascal identity,
\[
\begin{bmatrix}
\alpha\\ k
\end{bmatrix}_q
=
\begin{bmatrix}
\alpha-1\\ k
\end{bmatrix}_q
+
q^{\alpha-k}
\begin{bmatrix}
\alpha-1\\ k-1
\end{bmatrix}_q,
\]
and hence the bracketed expression is
\[
\rr_{\alpha}(x,y;u\,|\,q)=F_\alpha(x,y).
\]
This proves the Eq.(\ref{eqn_gpfde}). The initial value follows immediately from
the definition, since the constant term of \(F_\alpha\) is \(x^\alpha\).
\end{proof}

Equivalently, using
\[
D_qf(x)=\frac{f(x)-f(qx)}{x},
\]
the Eq.(\ref{eqn_gpfde}) may be written as
\[
uxF_\alpha(x,u^{-1}y)-uxF_\alpha(x,qu^{-1}y)
+
q^\alpha yF_\alpha(x,q^{-1}y)
-
yF_\alpha(x,y)
=
0.
\]
Replacing \(y\) by \(uqy\), we obtain the proportional
functional-difference equation

\begin{equation}\label{eqn_PFD}
xF_\alpha(x,qy)-xF_\alpha(x,q^2y)
+
q^{\alpha+1}yF_\alpha(x,uy)
-
qyF_\alpha(x,uqy)
=
0.
\end{equation}
Together with the operator realization of Section~4, this shows that the same family is both generated by a deformed \(q\)-translation operator and characterized by proportional functional relations coupling the scales \(y\), \(qy\), \(uy\), and \(uqy\). We record three elementary consequences illustrating the polynomial, negative-index, and specialized deformation regimes of the Eq.(\ref{eqn_PFD}).
\begin{example}
For \(\alpha=1\), one has
\[
F_1(x,y)=\rr_1(x,y;u\,|\,q)=x+y.
\]
The equation
\[
xF_1(x,qy)-xF_1(x,q^2y)+q^2yF_1(x,uy)-qyF_1(x,uqy)=0
\]
becomes
\[
x(x+qy)-x(x+q^2y)+q^2y(x+uy)-qy(x+uqy)=0,
\]
which is identically satisfied.
\end{example}

\begin{example}
For \(\alpha=-1\), the negative-index representation gives
\[
F_{-1}(x,y)
=
\rr_{-1}(x,y;u\,|\,q)
=
\Theta\left(-\frac{y}{qx},\frac{u}{q}\right).
\]
Hence the pantograph-type equation becomes
\[
x\Theta\left(-\frac{y}{x},\frac{u}{q}\right)
-
x\Theta\left(-\frac{qy}{x},\frac{u}{q}\right)
+
y\Theta\left(-\frac{uy}{qx},\frac{u}{q}\right)
-
qy\Theta\left(-\frac{uy}{x},\frac{u}{q}\right)
=0.
\]
Thus the same functional equation governing the deformed homogeneous
functions yields a nontrivial identity for Ramanujan partial theta
functions.
\end{example}

\begin{example}
If \(u=q\), then the equation reduces to
\[
xF_\alpha(x,qy)-xF_\alpha(x,q^2y)
+
q^{\alpha+1}yF_\alpha(x,qy)
-
qyF_\alpha(x,q^2y)=0,
\]
or equivalently
\begin{equation}\label{eqn_prod_alpha}
(x+q^{\alpha+1}y)F_\alpha(x,qy)
=
(x+qy)F_\alpha(x,q^2y).
\end{equation}
This specialization gives a first-order proportional functional
relation for \(\rr_\alpha(x,y;q\,|\,q)\).
\end{example}

The occurrence of a partial theta function in the case \(\alpha=-1\) is the first instance of a general phenomenon for all negative integer orders. This is developed in the next section.

\section{Negative indices and Ramanujan partial theta functions}

\subsection{The negative-index sector}

The negative-index sector is the first place where the complex-order
continuation departs substantially from the polynomial theory. While
nonnegative integer orders give finite deformed homogeneous polynomials,
negative integer orders give infinite \(q\)-series. The main result of
this section shows that these infinite series are not arbitrary: they
are governed by finite triangular combinations of Ramanujan partial theta functions.

For \(|Q|<1\), Ramanujan's partial theta function is defined by \[ \Theta(z,Q)=\sum_{k=0}^{\infty}Q^{\binom{k}{2}}z^k. \] In the sequel, we shall also use the degenerate boundary value corresponding to \(Q=1\). In this case, \[ \Theta(z,1)=\sum_{k=0}^{\infty}z^k=\frac{1}{1-z}, \qquad |z|<1. \] Thus, when the specialization \(u=q\) is considered, the partial theta terms \(\Theta(z,u/q)\) reduce to geometric series. This allows the Cauchy specialization to be treated within the same triangular framework, provided the corresponding geometric series converge.

\subsection{Finite theta decomposition}

We now establish the central results of the paper, which provide an explicit connection between $u$‑deformed homogeneous functions of negative index and Ramanujan’s partial theta function.

\begin{theorem}\label{theo_rr_theta}
Let \(n\in\mathbb N_0\) and $x\neq0$. As a formal power series in \(y\), one has
\[
\rr_{\minus n\minus1}(x,y;u\,|\,q)
=
\frac{x^{-n-1}}{(q;q)_n}
\sum_{i=0}^{n}
(-1)^i
\begin{bmatrix}
n\\ i
\end{bmatrix}_q
q^{\binom{i+1}{2}}
\Theta\left(
-\frac{y}{q^{n+1-i}x},
\frac{u}{q}
\right).
\]
If \(|u|<q\), the identity holds analytically in \(y\). In particular, when \(u=q\), the identity holds analytically for 
\[ 
\left|\frac{y}{x}\right|<q^{n+1}, 
\] 
after interpreting \(\Theta(z,1)\) as the geometric series \((1-z)^{-1}\).
\end{theorem}
\begin{proof}
For \(n=0\), the formula follows from Eq.(\ref{rr_GF_neg}) with \(\alpha=1\):
\[
\rr_{-1}(x,y;u\,|\,q)
=
x^{-1}
\Theta\left(-\frac{y}{qx},\frac{u}{q}\right).
\]
From Proposition 2 with \(\alpha=-1\), we have
\[
D_q^n \rr_{\minus1}(x,y;u\,|\,q)
=
(q;q)_n
\begin{bmatrix}
-1\\ n
\end{bmatrix}_q
\rr_{\minus n\minus1}(x,y;u\,|\,q).
\]
Since
\[
\begin{bmatrix}
-1\\ n
\end{bmatrix}_q
=
(-1)^n q^{-\binom{n+1}{2}},
\]
we obtain
\[
\rr_{\minus n\minus1}(x,y;u\,|\,q)
=
\frac{(-1)^n q^{\binom{n+1}{2}}}{(q;q)_n}
D_q^n \rr_{\minus1}(x,y;u\,|\,q).
\]
Then
\begin{align*}
    \rr_{\minus n\minus1}(x,y;u\,|\,q)&=\frac{(-1)^{n}q^{\binom{n+1}{2}}}{(q;q)_{n}}D_{q}^n\left\{\frac{1}{x}\Theta\left(-\frac{y}{qx},\frac{u}{q}\right)\right\}
\end{align*}
and by Mansour's identity, Eq.(\ref{mansour})
\begin{align*}
    &\rr_{\minus n\minus1}(x,y;u\,|\,q)\\
    &=\frac{(-1)^{n}q^{\binom{n+1}{2}}}{(q;q)_{n}}(-1)^nx^{-n}q^{-\binom{n}{2}}\sum_{i=0}^{n}(-1)^i\qbinom{n}{i}_{q}\frac{q^{\binom{i}{2}}}{xq^{n-i}}\Theta\left(-\frac{y}{xq^{1+n-i}},\frac{u}{q}\right)\\
    &=\frac{1}{(q;q)_{n}}x^{-n-1}\sum_{i=0}^{n}(-1)^i\qbinom{n}{i}_{q}q^{\binom{i+1}{2}}\Theta\left(-\frac{y}{xq^{1+n-i}},\frac{u}{q}\right).
\end{align*}
\end{proof}

\begin{remark}
This representation shows that the negative-index sector is not merely an infinite \(q\)-series extension of the polynomial theory; it is controlled by a finite system of shifted partial theta functions.
\end{remark}

\subsection{Triangular inversion}

\begin{corollary}[Triangular inversion]
For every \(n\in\mathbb N_0\),
\[
\Theta\left(
-\frac{y}{q^{n+1}x},
\frac{u}{q}
\right)
=
\sum_{k=0}^{n}
q^k
\begin{bmatrix}
n\\ k
\end{bmatrix}_q
(q;q)_{n-k}
x^{n-k+1}
R_{-n+k-1}(x,y;u\,|\,q).
\]
\end{corollary}
\begin{proof}
Set \[ A_n=(q;q)_n x^{n+1}R_{-n-1}(x,y;u\,|\,q), \qquad T_n=\Theta\left(-\frac{y}{q^{n+1}x},\frac{u}{q}\right). \] Then Theorem~3 gives \[ A_n=\sum_{j=0}^n (-1)^j \begin{bmatrix}n\\j\end{bmatrix}_q q^{\binom{j+1}{2}} T_{n-j}. \] The stated formula follows from the corresponding triangular \(q\)-binomial inversion.
\end{proof}

Consequently, the connection established in Theorem 3 is not merely a representation formula. Together with Corollary 2, it gives a finite triangular equivalence between the negative-index sector of the \(u\)-deformed homogeneous functions and a shifted system of Ramanujan partial theta functions.

\subsection{Examples}

The first cases make the triangular pattern transparent.
\begin{example}
\[
\rr_{\minus2}(x,y;u\,|\,q)=\frac{x^{-2}}{1-q}\left(\Theta\left(-\frac{y}{q^2x},\frac{u}{q}\right)-q\Theta\left(-\frac{y}{qx},\frac{u}{q}\right)\right)
\]

\[
\rr_{\minus3}(x,y;u\,|\,q)=\frac{x^{-3}}{(q;q)_{2}}\bigg(\Theta\left(-\frac{y}{q^3x},\frac{u}{q}\right)
    -(1+q)q\Theta\left(-\frac{y}{q^2x},\frac{u}{q}\right)+q^3\Theta\left(-\frac{y}{qx},\frac{u}{q}\right)\bigg)
\]

\begin{align*}
    \rr_{\minus4}(x,y;u\,|\,q)&=\frac{x^{-4}}{(q;q)_{3}}\bigg(\Theta\left(-\frac{y}{q^4x},\frac{u}{q}\right)-(1+q+q^2)q\Theta\left(-\frac{y}{q^3x},\frac{u}{q}\right)\\
    &\hspace{2cm}+(1+q+q^2)q^3\Theta\left(-\frac{y}{q^2x},\frac{u}{q}\right)-q^6\Theta\left(-\frac{y}{qx},\frac{u}{q}\right)\bigg)
\end{align*}
\end{example}
The inverse formulas begin as follows.
\begin{example}
\[
\Theta\left(-\frac{y}{qx},\frac{u}{q}\right)=x\rr_{-1}(x,y;u\,|\,q)
\]

\[
\Theta\left(-\frac{y}{q^2x},\frac{u}{q}\right)=(1-q)x^2\rr_{-2}(x,y;u\,|\,q)+qx\rr_{-1}(x,y;u\,|\,q).
\]
    
\end{example}

\subsection{Asymptotic limit}

The finite theta decomposition admits a limiting form when the negative index tends to infinity and the second variable is simultaneously scaled by a power of \(q\).
\begin{theorem}
If $0<\vert u\vert<q$, then
\begin{equation}
    \lim_{n\rightarrow\infty}x^{n+1}\rr_{\minus n\minus1}(x,q^ny;u\,|\,q)=\frac{1}{(q;q)_{\infty}}\sum_{i=0}^{\infty}(-1)^i\frac{q^{\binom{i+1}{2}}}{(q;q)_{i}}\Theta\left(-q^{i-1}\frac{y}{x},\frac{u}{q}\right).
\end{equation}
\end{theorem}
\begin{proof}
As
\begin{equation*}
    \lim_{n\rightarrow\infty}\qbinom{n}{i}_{q}=\frac{1}{(q;q)_{i}},
\end{equation*}
then from Theorem \ref{theo_rr_theta},
\begin{align*}
    &\lim_{n\rightarrow\infty}x^{n+1}\rr_{\minus n\minus1}(x,q^ny;u\,|\,q)\\
    &\hspace{1cm}=\lim_{n\rightarrow\infty}\frac{1}{(q;q)_{n}}\sum_{i=0}^{n}(-1)^i\qbinom{n}{i}_{q}q^{\binom{i+1}{2}}\Theta\left(-q^{i-1}\frac{y}{x},\frac{u}{q}\right)\\
    &\hspace{1cm}=\frac{1}{(q;q)_{\infty}}\sum_{i=0}^{\infty}(-1)^i\frac{q^{\binom{i+1}{2}}}{(q;q)_{i}}\Theta\left(-q^{i-1}\frac{y}{x},\frac{u}{q}\right).
\end{align*}  
Since \(|u/q|<1\), the function \(\Theta(z,u/q)\) is entire in \(z\). On compact subsets of the \(y\)-plane, the arguments \(-q^{i-1}y/x\) remain in a fixed compact set. Hence the theta factors are uniformly bounded. Moreover, \[ \frac{q^{\binom{i+1}{2}}}{(q;q)_i} \] is summable, since \((q;q)_i\ge (q;q)_\infty>0\). Therefore, dominated convergence applies.
\end{proof}

\section{Specializations}

\subsection{Polynomial specializations}

The general family \(\rr_{\alpha}(x,y;u\,|\,q)\) contains, for suitable choices of the deformation parameter \(u\), several classical and deformed polynomial families. When \(\alpha=n\in\mathbb N_0\), the series terminates and these specializations reduce to finite polynomial systems. The purpose of this section is to show that the negative-index theta decompositions obtained in Section 6 also admit corresponding specialized forms.

\subsection{Cauchy specialization}

We begin with the specialization \(u=q\). In this case the deformation
factor becomes \(q^{\binom{k}{2}}\), and we define the Cauchy-type
function by
\[
\P_{\alpha}(x,y\,|\,q)
=
\rr_{\alpha}(x,-y;q\,|\,q).
\]
Thus
\begin{equation}\label{cauchy_alpha}
\P_{\alpha}(x,y\,|\,q)
=
\sum_{k=0}^{\infty}
\begin{bmatrix}
\alpha\\ k
\end{bmatrix}_{q}
(-1)^kq^{\binom{k}{2}}x^{\alpha-k}y^k.
\end{equation}
By Theorem~1, if \(\alpha\notin\mathbb N_0\), this series converges for \[ \left|\frac{y}{x}\right|<|q^{-\alpha}|. \] For \(\alpha\in\mathbb N_0\), it terminates and one obtains
the Cauchy-type polynomials
\[
\P_n(x,y\,|\,q)
=
\sum_{k=0}^{n}
\begin{bmatrix}
n\\ k
\end{bmatrix}_{q}
(-1)^kq^{\binom{k}{2}}x^{n-k}y^k.
\]
By the \(q\)-binomial theorem, these polynomials factor as
\[
\P_n(x,y\,|\,q)
=
x^n\left(\frac{y}{x};q\right)_n
=
\prod_{j=0}^{n-1}(x-q^j y).
\]
This shows that the specialization \(u=q\) recovers the finite
\(q\)-Cauchy product in the polynomial regime. By the \(q\)-binomial theorem, the non-polynomial Cauchy specialization admits the infinite product representation
\begin{equation}\label{eqn_fcauchy}
    \P_{\alpha}(x,y|q)=x^\alpha\frac{(y/x;q)_\infty}{(q^\alpha y/x;q)_\infty}\qquad \left|\frac{y}{x}\right|<|q^{-\alpha}|.
\end{equation}

For non-polynomial orders, the specialization $u=q$ lies on the boundary of the partial-theta parameter appearing in the negative-index theory. Indeed, the finite theta decompositions of Section~6 involve the parameter \(u/q\), and for \(u=q\) one has \(u/q=1\). Consequently,
\[
\Theta(z,1)
=
\sum_{k=0}^{\infty}z^k
=
\frac{1}{1-z},
\qquad |z|<1.
\]
Thus the Cauchy specialization transforms the partial-theta
decompositions into finite rational decompositions.

\begin{corollary}[Negative-index Cauchy decomposition] 
Let \(n\in\mathbb N_0\) and \(x\neq 0\). Then 
\[ 
\P_{\minus n\minus1}(x,y\,|\,q) = \frac{x^{-n-1}}{(q;q)_n} \sum_{i=0}^{n} (-1)^i \qbinom{n}{i}_{q} q^{\binom{i+1}{2}} \Theta\left( \frac{y}{q^{n+1-i}x},1 \right). 
\] 
Equivalently, 
\[ 
\P_{\minus n\minus1}(x,y\,|\,q) = \frac{x^{-n-1}}{(q;q)_n} \sum_{i=0}^{n} (-1)^i \qbinom{n}{i}_{q} q^{\binom{i+1}{2}} \frac{q^{n+1-i}x}{q^{n+1-i}x-y}, 
\] 
whenever 
\[ 
\left|\frac{y}{x}\right|<q^{n+1}. 
\]
\end{corollary}
\begin{proof} 
This follows from Theorem~3 by replacing \(y\) with \(-y\) and setting \(u=q\). Since \(u/q=1\), the partial theta functions reduce to geometric series: \[ \Theta(z,1)=\frac{1}{1-z}. \] The common convergence condition is obtained from \[ \left| \frac{y}{q^{n+1-i}x} \right|<1, \qquad i=0,1,\ldots,n. \] The most restrictive condition occurs when \(i=0\), giving \[ \left|\frac{y}{x}\right|<q^{n+1}. \] 
\end{proof}

\begin{corollary}[Inverse Cauchy transform] 
For every \(n\in\mathbb N_0\), 
\[ 
\frac{q^{n+1}x}{q^{n+1}x-y} = \sum_{k=0}^{n} q^k\qbinom{n}{k}_q(q;q)_{n-k} x^{n-k+1}\P_{\minus n+k\minus1}(x,y\,|\,q), 
\] 
for 
\[ 
\left|\frac{y}{x}\right|<q^{n+1}. 
\] 
\end{corollary}
\begin{proof} 
This is the specialization of Corollary~2 obtained by replacing \(y\) with \(-y\) and setting \(u=q\). Since \(u/q=1\), the shifted partial theta functions become geometric series. 
\end{proof}

The first inverse identities illustrate how simple rational functions are recovered from negative-index Cauchy functions
\[ 
\frac{qx}{qx-y} = x\P_{-1}(x,y\,|\,q)
\] 
and 
\[ 
\frac{q^2x}{q^2x-y} = (1-q)x^2\P_{-2}(x,y\,|\,q) + qx\P_{-1}(x,y\,|\,q). 
\]
\begin{corollary}[Cauchy asymptotic limit] 
Let \(x\neq0\). Then 
\[ 
\lim_{n\to\infty} x^{n+1}\P_{\minus n\minus1}(x,q^ny\,|\,q) = \frac{1}{(q;q)_\infty} \sum_{i=0}^{\infty} \frac{(-1)^i q^{\binom{i+1}{2}}}{(q;q)_i} \frac{x}{x-q^{i-1}y}. 
\] 
The convergence is locally uniform for \(\left|y/x\right|<q\). 
\end{corollary}
\begin{proof} 
The proof follows the same limiting argument as Theorem~4, after replacing \(y\) with \(-y\) and setting \(u=q\). Let \(K\) be a compact subset of 
\[ 
\left\{ y:\left|\frac{y}{x}\right|<q \right\}. 
\] 
Then there exists \(\rho<q\) such that 
\[ 
\left|\frac{y}{x}\right|\leq \rho 
\] 
for all \(y\in K\). Hence, for every \(i\geq0\), 
\[ 
\left|q^{i-1}\frac{y}{x}\right| \leq q^{i-1}\rho \leq \frac{\rho}{q}<1. 
\] 
Therefore the geometric factors 
\[ 
\Theta\left(q^{i-1}\frac{y}{x},1\right) = \frac{1}{1-q^{i-1}y/x} 
\] 
are uniformly bounded on \(K\). Since 
\[ 
\sum_{i=0}^{\infty} \frac{q^{\binom{i+1}{2}}}{(q;q)_i} 
\] 
converges, dominated convergence applies. This gives the stated limit. 
\end{proof}

Thus the Cauchy specialization represents the boundary case of the partial-theta framework. In this case the theta functions degenerate into geometric series, and the finite triangular theta correspondence becomes a finite triangular rational correspondence.

\subsection{Stieltjes--Wigert-type specializations}

We now consider the specialization \[ u=q^b,\qquad b\geq 2. \] The restriction \(b\geq2\) is imposed in order to exclude the boundary case \(b=1\), which corresponds to \(u=q\) and was treated separately in the Cauchy specialization. Indeed, the partial-theta parameter appearing in the negative-index decompositions is \(u/q\). Hence, for \(u=q^b\), this parameter becomes \[ \frac{u}{q}=q^{b-1}. \] When \(b=1\), one obtains \(q^{b-1}=1\), so that the partial theta functions degenerate into geometric series, precisely as in the Cauchy case. For \(b\geq2\), however, one has \(0<q^{b-1}<1\), and the specialized decompositions involve genuine Ramanujan partial theta functions.

In order to recover the usual Stieltjes--Wigert-type sign pattern in the polynomial case, we introduce 
\[ 
\s_{\alpha,b}(x,y;q) = \rr_{\alpha}(x,-qy;q^b\,|\,q). 
\] 
Thus 
\[ 
\s_{\alpha,b}(x,y;q) = \sum_{k=0}^{\infty} \qbinom{\alpha}{k}_{q} (-1)^k q^{b\binom{k}{2}+k} x^{\alpha-k}y^k . 
\] 
In particular, when \(b=2\), we write 
\[ 
\s_{\alpha}(x,y;q):=\s_{\alpha,2}(x,y;q), 
\] 
so that 
\[ 
\s_{\alpha}(x,y;q) = \sum_{k=0}^{\infty}\qbinom{\alpha}{k}_{q} (-1)^k q^{k^2}x^{\alpha-k}y^k . 
\] 
For \(\alpha=n\in\mathbb N_0\), the series terminates and gives the Stieltjes--Wigert-type polynomials 
\[ 
\s_{n,b}(x,y;q) = \sum_{k=0}^{n} \qbinom{n}{k}_{q} (-1)^k q^{b\binom{k}{2}+k} x^{n-k}y^k . 
\] 
The case \(b=2\) corresponds to the standard \(q^{k^2}\)-weight.
From the negative-order expansion of \(\rr_{\alpha}\), one obtains 
\[ 
\s_{-\alpha,b}(x,y;q) = \frac{1}{x^\alpha} \sum_{k=0}^{\infty} \begin{bmatrix} \alpha+k-1\\ k \end{bmatrix}_{q} q^{(b-1)\binom{k}{2}} \left( \frac{y}{q^{\alpha-1}x} \right)^k . 
\] 
In particular, 
\[ 
\s_{-n,b}(x,y;q) = x^{-n} \sum_{k=0}^{\infty} \begin{bmatrix} n+k-1\\ k \end{bmatrix}_{q} q^{(b-1)\binom{k}{2}} \left( \frac{y}{q^{n-1}x} \right)^k . 
\]

\begin{corollary}[Negative-index Stieltjes--Wigert-type decomposition]
Let \(b\geq2\), \(n\in\mathbb N_0\), and \(x\neq0\). Then
\[
\s_{-n-1,b}(x,y;q)
=
\frac{x^{-n-1}}{(q;q)_n}
\sum_{i=0}^{n}
(-1)^i
\begin{bmatrix}
n\\ i
\end{bmatrix}_{q}
q^{\binom{i+1}{2}}
\Theta\left(
\frac{y}{q^{\,n-i}x},
q^{b-1}
\right).
\]
The identity holds analytically in \(y\).
\end{corollary}
\begin{proof}
This follows from Theorem~3 by replacing \(y\) with \(-qy\) and setting
\(u=q^b\). Then the partial-theta parameter becomes
\[
\frac{u}{q}=q^{b-1}.
\]
For \(b\geq2\), one has \(|q^{b-1}|<1\), so the theta functions are genuine partial theta functions. 
\end{proof}

In the classical Stieltjes--Wigert-type case \(b=2\), this gives 
\[ 
\s_{\minus n\minus1}(x,y;q) = \frac{x^{-n-1}}{(q;q)_n} \sum_{i=0}^{n} (-1)^i \begin{bmatrix} n\\ i \end{bmatrix}_{q} q^{\binom{i+1}{2}} \Theta\left( \frac{y}{q^{\,n-i}x}, q \right). 
\]

\begin{corollary}[Inverse Stieltjes--Wigert-type transform]
For every \(b\geq2\) and \(n\in\mathbb N_0\),
\[
\Theta\left(
\frac{y}{q^n x},
q^{b-1}
\right)
=
\sum_{k=0}^{n}
q^k
\begin{bmatrix}
n\\ k
\end{bmatrix}_{q}
(q;q)_{n-k}
x^{n-k+1}
\s_{-n+k-1,b}(x,y;q).
\]
\end{corollary}
\begin{proof}
This follows from Corollary~2 by replacing \(y\) with \(-qy\) and
setting \(u=q^b\). Under this specialization,
\[
-\frac{-qy}{q^{n+1}x}
=
\frac{y}{q^n x},
\qquad
\frac{u}{q}=q^{b-1}.
\]
\end{proof}

The first inverse identities are 
\[ 
\Theta\left(\frac{y}{x},q^{b-1}\right) = x\s_{-1,b}(x,y;q), 
\] 
and 
\[ 
\Theta\left(\frac{y}{qx},q^{b-1}\right) = (1-q)x^2\s_{-2,b}(x,y;q) + qx\s_{-1,b}(x,y;q). 
\] 
For \(b=2\), these become \[ \Theta\left(\frac{y}{x},q\right) = x\s_{-1}(x,y;q), 
\] 
and 
\[ 
\Theta\left(\frac{y}{qx},q\right) = (1-q)x^2\s_{-2}(x,y;q) + qx\s_{-1}(x,y;q). 
\]

\begin{corollary}[Stieltjes--Wigert-type asymptotic limit]
Let \(b\geq2\) and \(x\neq0\). Then
\[
\lim_{n\to\infty}
x^{n+1}\s_{-n-1,b}(x,q^ny;q)
=
\frac{1}{(q;q)_\infty}
\sum_{i=0}^{\infty}
\frac{(-1)^iq^{\binom{i+1}{2}}}{(q;q)_i}
\Theta\left(
q^i\frac{y}{x},
q^{b-1}
\right).
\]
The convergence is locally uniform in \(y/x\).
\end{corollary}
\begin{proof}
This follows from Theorem~4 by replacing \(y\) with \(-qy\) and setting
\(u=q^b\). Since \(b\geq2\), one has \(0<q^b<q\), so the hypothesis of
Theorem~4 is satisfied. Moreover,
\[
-\frac{q^{i-1}(-qy)}{x}
=
q^i\frac{y}{x},
\qquad
\frac{u}{q}=q^{b-1}.
\]
Therefore the stated limit follows directly.
\end{proof}

Thus the Stieltjes--Wigert-type specialization lies inside the analytic partial-theta regime whenever \(b\geq2\). In contrast with the Cauchy specialization, where the theta parameter reaches the boundary value \(1\), the Stieltjes--Wigert-type family produces genuine partial theta functions with parameter \(q^{b-1}\).

\subsection{Further half-integer specializations}

We briefly mention that the same framework also covers half-integer deformation parameters. More precisely, for 
\[ 
u=q^{b+\frac12},\qquad b\geq1, 
\] 
one obtains the Exton-type specialization 
\[ 
\E_{\alpha,b}(x,y;q) = \rr_{\alpha}\left(x,y;q^{b+\frac12}\,|\,q\right). 
\] 
Thus 
\[ 
\E_{\alpha,b}(x,y;q) = \sum_{k=0}^{\infty} \begin{bmatrix} \alpha\\ k \end{bmatrix}_{q} q^{\left(b+\frac12\right)\binom{k}{2}} x^{\alpha-k}y^k. 
\] 
For \(\alpha=n\in\mathbb N_0\), this reduces to a finite Exton-type polynomial family \cite{orozco}. From the point of view of the negative-index theory, this specialization belongs entirely to the analytic partial-theta regime. Indeed, the partial-theta parameter appearing in Theorem~3 and Corollary~2 is \(u/q\), and in the present case 
\[ 
\frac{u}{q}=q^{b-\frac12}. 
\] 
Since \(b\geq1\) and \(0<q<1\), one has 
\[ 
0<q^{b-\frac12}<1. 
\] 
Consequently, Theorem~3, Corollary~2, and Theorem~4 specialize directly to the Exton-type family \(\E_{\alpha,b}(x,y;q)\), with the partial theta parameter \(q^{b-\frac12}\). We do not write these formulas explicitly, since they are obtained from the preceding results by the immediate substitution \[ 
u=q^{b+\frac12}. 
\] 
Thus the half-integer Exton-type regime provides a natural companion to the integer Stieltjes--Wigert-type regime, replacing the theta parameter \(q^{b-1}\) by the half-integer parameter \(q^{b-\frac12}\).

\subsection{Remarks on the specialized negative-index sector}

The specializations considered above illustrate three different regimes of the negative-index theory, distinguished by the value of the partial-theta parameter \(u/q\). First, the Cauchy specialization corresponds to the boundary case \[ u=q, \qquad \frac{u}{q}=1. \] In this regime the partial theta function degenerates into the geometric series \[ \Theta(z,1)=\frac{1}{1-z}, \] inside its disk of convergence. Consequently, the finite triangular partial-theta decompositions of the negative-index functions become finite triangular rational decompositions. Second, the Stieltjes--Wigert-type specialization corresponds to \[ u=q^b,\qquad b\geq2, \] and hence \[ \frac{u}{q}=q^{b-1}. \] Since \(0<q^{b-1}<1\), this specialization belongs to the genuine partial-theta regime. Thus the negative-index Stieltjes--Wigert-type functions are governed by finite triangular combinations of shifted Ramanujan partial theta functions with theta parameter \(q^{b-1}\). Finally, the half-integer Exton-type specialization \[ u=q^{b+\frac12},\qquad b\geq1, \] gives \[ \frac{u}{q}=q^{b-\frac12}, \] which also lies strictly inside the unit disk. Hence this regime is covered by the same finite triangular mechanism, with theta parameter \(q^{b-\frac12}\). Since the corresponding formulas follow directly from Theorem~3, Corollary~2, and Theorem~4 by the substitution \(u=q^{b+\frac12}\), no separate derivation is needed. Thus the specialized negative-index sector separates naturally into a boundary geometric case and genuine partial-theta cases. The Cauchy specialization produces rational triangular decompositions, whereas the Stieltjes--Wigert and Exton-type specializations remain within the analytic partial-theta framework. This shows that the triangular correspondence established in Section~6 is stable under classical \(q\)-special regimes and is not tied to a single choice of the deformation parameter.

\section{Further directions} 

The results of this paper suggest several directions for further investigation. We list some of them below. 
\begin{enumerate} 
\item \textbf{Analytic dependence on the complex order.} In the present work, we have focused mainly on convergence in the second variable, recurrence relations, \(q\)-differentiation formulas, and the negative-index sector. A natural continuation is to study the dependence of 
\[ 
\rr_{\alpha}(x,y;u\,|\,q) 
\] 
on the complex parameter \(\alpha\). This includes questions about analytic continuation, possible singularities, growth properties, and the behavior of zeros in the non-polynomial regime. 

\item \textbf{A transform theory for the negative-index sector.} Theorem~3 and Corollary~2 show that the negative-index functions and shifted Ramanujan partial theta functions are related by an explicit finite triangular transform. It would be interesting to develop a more systematic transform theory associated with this triangular correspondence. In particular, one may ask whether natural dual systems, moment functionals, or biorthogonality relations can be constructed from this equivalence. 

\item \textbf{Further operator-theoretic developments.} The operator realization in Section~4 shows that \(\rr_{\alpha}(x,y;u\,|\,q)\) arises from the action of a deformed \(q\)-exponential operator on complex powers. A broader operator calculus could be developed by studying the action of this operator on other classes of functions, its interaction with Jackson \(q\)-derivatives, and its possible role in generating additional identities for \(q\)-special functions. 

\item \textbf{Additional specializations.} The specializations discussed in Section~7 show that the negative-index theory is stable under several classical \(q\)-special regimes. The Cauchy specialization corresponds to a boundary case in which partial theta functions degenerate into geometric series, whereas the Stieltjes--Wigert and half-integer Exton-type regimes remain inside the genuine partial-theta domain. Other choices of the deformation parameter \(u\), especially those related to known \(q\)-polynomial families, may lead to further triangular decompositions and limiting formulas. 

\item \textbf{Combinatorial and asymptotic interpretations.} Since Ramanujan partial theta functions arise naturally from the negative orders of the deformed homogeneous calculus, it would be worth investigating whether the triangular decompositions obtained here admit combinatorial interpretations. Possible directions include partition-theoretic generating functions, asymptotic analysis of the negative-index sector, and related questions in the theory of \(q\)-series. 

\item \textbf{Possible modular or mock-theta connections.} Partial theta functions are closely related to several themes in the broader theory of \(q\)-series. Although no modular or mock-modular properties are proved in this paper, the appearance of partial theta functions in the negative-index sector suggests that such connections may deserve further investigation. 
\end{enumerate} 
These directions indicate that the negative-index sector of the \(u\)-deformed homogeneous calculus is not merely a formal extension of the polynomial theory, but a natural source of partial-theta structures and triangular transformations.

\end{document}